\documentclass[11pt]{amsart}

\usepackage{pstricks,graphicx,pst-node}
\newtheorem{theorem}{Theorem}[section]
\newtheorem{lemma}[theorem]{Lemma}
\newtheorem{proposition}[theorem]{Proposition}
\newtheorem{corollary}[theorem]{Corollary}
\newtheorem{conjecture}[theorem]{Conjecture}

\theoremstyle{definition}
\newtheorem{definition}[theorem]{Definition}

\theoremstyle{remark}
\newtheorem{remark}[theorem]{Remark}

\newcommand{\reffig}[1]{Figure \ref{fig:#1}}
\newcommand{\refeq}[1]{equation (\ref{eqn:#1})}

\newcommand{\refthm}[1]{Theorem \ref{thm:#1}}
\newcommand{\reflem}[1]{Lemma \ref{lem:#1}}
\newcommand{\refprop}[1]{Proposition \ref{prop:#1}}

\newcommand{\refrmk}[1]{Remark \ref{rmk:#1}}

\newcommand{\refsec}[1]{Section \ref{sec:#1}}

\newcommand{\inv}{\ensuremath\mathrm{inv}}
\newcommand{\maj}{\ensuremath\mathrm{maj}}
\newcommand{\Inv}{\ensuremath\mathrm{Inv}}
\newcommand{\Des}{\ensuremath\mathrm{Des}}
\newcommand{\iDes}{\ensuremath\mathrm{iDes}}

\newcommand{\SYT}{\ensuremath\mathrm{SYT}}

\newcommand{\dist}{\mathrm{dist}}
\newcommand{\ww}{\widetilde{w}}

\newlength{\hsp}
\newlength{\vsp}
\newlength{\vspi}
\newcommand{\ccr}  {c@{\hskip 0.707\hsp}} 
\newcommand{\ccii} {c@{\hskip 2\hsp}} 
\newcommand{\cciii}{c@{\hskip 3\hsp}} 

\newcommand{\rn}{\rnode}

\newlength\cellsize 

\savebox2{%
\begin{picture}(12,12)
\put(0,0){\line(1,0){12}}
\put(0,0){\line(0,1){12}}
\put(12,0){\line(0,1){12}}
\put(0,12){\line(1,0){12}}
\end{picture}}

\newcommand\cellify[1]{\def\thearg{#1}\def\nothing{}%
\ifx\thearg\nothing
\vrule width0pt height\cellsize depth0pt\else
\hbox to 0pt{\usebox2\hss}\fi%
\vbox to 12\unitlength{
\vss
\hbox to 12\unitlength{\hss$#1$\hss}
\vss}}

\newcommand\tableau[1]{\vtop{\let\\=\cr
\setlength\baselineskip{-12000pt}
\setlength\lineskiplimit{12000pt}
\setlength\lineskip{0pt}
\halign{&\cellify{##}\cr#1\crcr}}}

\newlength\bigcellsize 

\savebox3{%
\begin{picture}(20,20)
\put(0,0){\line(1,0){20}}
\put(0,0){\line(0,1){20}}
\put(20,0){\line(0,1){20}}
\put(0,20){\line(1,0){20}}
\end{picture}}

\newcommand\bigcellify[1]{\def\thearg{#1}\def\nothing{}%
\ifx\thearg\nothing
\vrule width0pt height\bigcellsize depth0pt\else
\hbox to 0pt{\usebox3\hss}\fi%
\vbox to 20\unitlength{
\vss
\hbox to 20\unitlength{\hss$#1$\hss}
\vss}}

\newcommand\bigtableau[1]{\vtop{\let\\=\cr
\setlength\baselineskip{-16000pt}
\setlength\lineskiplimit{16000pt}
\setlength\lineskip{0pt}
\halign{&\bigcellify{##}\cr#1\crcr}}}

\begin{document}


\title[Generalized major index]{A generalized major index statistic}

\author[S. Assaf]{Sami H. Assaf}
\address{Department of Mathematics, University of Pennsylvania,
Philadelphia, PA 19104}
\email{sassaf@math.mit.edu}
\thanks{Work supported by NSF Postdoctoral Fellowship DMS-0703567.}


\date{\today}



\begin{abstract}
  Inspired by the $k$-inversion statistic for LLT polynomials, we
  define a $k$-inversion number and $k$-descent set for words. Using
  these, we define a new statistic on words, called the $k$-major
  index, that interpolates between the major index and inversion
  number. We give a bijective proof that the $k$-major index is
  equidistributed with the major index, generalizing a classical
  result of Foata and rediscovering a result of Kadell. Inspired by
  recent work of Haglund and Stevens, we give a partial extension of
  these definitions and constructions to standard Young
  tableaux. Finally, we give an application to Macdonald polynomials
  made possible through connections with LLT polynomials.
\end{abstract}

\maketitle

\section{Introduction}
\label{sec:intro}

Given a multiset $M$ of $n$ positive integers, a word on $M$ is a
sequence of positive integers $w = w_1w_2 \cdots w_n$ that reorders
$M$. A {\em statistic on words} is an association of an element of
$\mathbb{N}$ to each word. A fundamental statistic that has been
rediscovered in many guises is the {\em inversion number} of a word,
defined as the number of pairs of indices $(i<j)$ such that $w_i >
w_j$. A {\em descent} of a word is an index $i$ such that $w_i >
w_{i+1}$. In 1913, Major P. MacMahon \cite{MacMahon1913} introduced an
important statistic, now called the {\em major index} in his honor,
defined as the sum over the descents of a word. Using generating
functions, MacMahon \cite{MacMahon1916} proved the remarkable fact
that the major index has the same distribution as the inversion
number. Precisely, he showed that for $W_M$ the set of words on a
fixed multiset $M$,
\begin{displaymath}
  \sum_{w \in W_M} q^{\maj(w)} = \sum_{w \in W_M} q^{\inv(w)},
\end{displaymath}
where $\maj(w)$ denotes the major index of $w$ and $\inv(w)$ denotes
the inversion number of $w$. Any statistic that is equidistributed
with the major index, i.e. a statistic satisfying the above equation,
is called {\em Mahonian}. MacMahon then raised the question to find a
bijective proof that the inversion number is Mahonian. This question
was first resolved by Foata \cite{Foata1968}, who constructed a
bijection on words with the property that the major index of a word
equals the inversion number of its image.

In this paper, we introduce a statistic called the $k$-major index
which interpolates between the major index and inversion number. More
precisely, the $1$-major index is MacMahon's major index, and the
$n$-major index of a word of length $n$ is the inversion number. By
constructing bijections on words with a recursive structure similar to
Foata's bijection, we give a bijective proof that the $k$-major index
is Mahonian for all $k$. Looking back through the literature, this
same statistic was discover by Kadell \cite{Kadell1985} who also gave
a bijective proof that the distribution is Mahonian. Whereas Kadell's
bijections in fact refine Foata's original bijection, the family of
bijections defined herein is not the same as Kadell's and, when taking
the major index to the inversion number, give a bijection different
from that of Foata. The $k$-major index statistic is defined in
\refsec{stats}, and the bijections and proof that the distribution is
Mahonian are given in \refsec{bijections}.

It is also natural to define a major index statistic on standard Young
tableaux, which are central objects in the study of symmetric
functions. Recently, Haglund and Stevens \cite{HaSt2006} defined an
inversion number on tableaux. Their construction generalizes Foata's
bijection to tableaux and shows that the inversion number and major
index are equidistributed over standard Young tableaux of a fixed
shape. Motivated by this, we use the bijections presented here to
extend the notion of the $k$-major index to standard Young tableaux,
for $k \leq 3$. The hope is that this method might be used to build a
complete family of statistics interpolating between major index and
inversion number on tableaux. This exploration takes place in
\refsec{tableau}.

Our discovery of the $k$-major index and the family of bijections
presented here came about through the study of Macdonald polynomials
\cite{Assaf2007-2}. In \refsec{macdonald}, we elaborate on this
connection and present a conjecture for yet another family of
bijections sharing many of the same properties that would have the
further consequence of providing a remarkably simple combinatorial
proof of Macdonald positivity.

\section{Definitions and notation}
\label{sec:stats}

At times it will be convenient to consider a slightly more general
definition for a word $w$, where $w_i$ is allowed to be either a
positive integer or an $\emptyset$. In this case, $\emptyset$'s should
be regarded as incomparable to other letters, so that they are simply
a way of spacing out the nonempty letters of $w$.  This idea will be
especially important in connection with Macdonald polynomials
discussed in \refsec{macdonald}.

\begin{definition}
  For $w$ a word, $k$ a positive integer, define the {\em $k$-descent
    set of $w$}, denoted $\Des_k(w)$, by
  \begin{displaymath}
    \Des_k(w) = \{ (i,i+k) \; | \; w_i > w_{i+k} \} ,
  \end{displaymath}
  and define the {\em $k$-inversion set of $w$}, denoted $\Inv_k(w)$,
  by
  \begin{displaymath}
    \Inv_k(w) = \{ (i,j) \; | \; k > j-i > 0 \; \mbox{and} \; w_i >
    w_j \} .
  \end{displaymath}
\end{definition}

For example, for $w = 986173245$ and $k=3$ we have
\begin{eqnarray*}
  \Des_3(9 \; 8 \; 6 \; 1 \; 7 \; 3 \; 2 \; 4 \; 5) & = & \{
  (1,4), (2,5), (3,6), (5,8)\}, \\ 
  \Inv_3(9 \; 8 \; 6 \; 1 \; 7 \; 3 \; 2 \; 4 \; 5) & = & \{ (1,2),
  (1,3), (2,3), (2,4), (3,4), (5,6), (5,7), (6,7)\}.
\end{eqnarray*}

In fact, it is enough to define $k$-descents since $k$-inversions may
be recovered from the observation
\begin{equation}
  \Inv_k(w) \; = \; \bigcup_{j < k} \Des_j(w) .
\label{eqn:alt-Invk}
\end{equation}

Note that when $k=1$, $\Des_k$ gives the usual descent set for a
word. Similarly, when $N \geq n$, $\Inv_N$ gives the usual set of
inversion pairs for a word of length $n$. We interpolate between the
corresponding statistics, $\maj$ and $\inv$, with the following
statistic depending on the parameter $k$.

\begin{definition}
  Given a word $w$ and a positive integer $k$, define the {\em
    $k$-major index of $w$} by
  \begin{displaymath}
    \maj_k (w) \; = \; \left| \Inv_k(w) \right| + \sum_{(i,i+k) \in
      \Des_k(w)} i .
  \end{displaymath}
\end{definition}  

For the same example, we have $\maj_3(9 \; 8 \; 6 \; 1 \; 7 \; 3 \; 2
\; 4 \; 5) = 8 + 1 + 2 + 3 + 5 = 19$. For a word $w$ of length $n \leq
N$, the previous observations show that
\begin{eqnarray*}
  \maj_1(w) & = & \maj(w), \\
  \maj_N(w) & = & \inv(w).
\end{eqnarray*}
The statistic $\maj_k$ was first defined by Kadell \cite{Kadell1985},
who gives a bijective proof that this statistic is Mahonian. Kadell's
bijections take $\inv$ to $\maj_k$, with the extreme case from $\inv$
to $\maj$ corresponding precisely to the inverse of Foata's bijection
\cite{Foata1968}. In \refsec{bijections}, we give a different family
of bijections, taking $\maj_{k-1}$ to $\maj_k$, which, when composed
appropriately, give a different bijection from $\maj$ to $\inv$.

In the case when $w$ is a permutation (possibly with $\emptyset$s), we
will also be interested in the descent set of the inverse permutation,
denoted $\iDes$, defined by
\begin{equation}
  \iDes(w) = \Des(w^{-1}) = \{ i \ | \ \mbox{$i$ appears to the left
    of $i+1$ in $w$} \} .
\end{equation}
For example, $\iDes(9 \; 8 \; 6 \; 1 \; 7 \; 3 \; 2 \; 4 \; 5) =
\{2,5,7,8\}$. 

Recall that a {\em partition} $\lambda$ is a weakly decreasing
sequence of positive integers:
$\lambda=(\lambda_1,\lambda_2,\ldots,\lambda_m)$, $\lambda_1 \geq
\lambda_2 \geq \cdots \geq \lambda_m > 0$. A partition $\lambda$ may
be identified with its {\em Young diagram}: the set of points $(i,j)$
in the $\mathbb{Z}_+ \times\mathbb{Z}_+$ lattice quadrant such that $1
\leq i \leq \lambda_j$. We draw the diagram so that each point $(i,j)$
is represented by the unit cell southwest of the point. A {\em
  standard Young tableau of shape $\lambda$} is a labelling of the
cells of the Young diagram of $\lambda$ with the numbers $1$ through
$n$, where $n = \sum_i \lambda_i$, such that the entries increase
along rows and up columns. For example, see \reffig{tableau}.

\begin{figure}[ht]
  \begin{center}
    \begin{displaymath}
      \tableau{8 \\ 2 & 5 & 6 \\ 1 & 3 & 4 & 7}
    \end{displaymath}
    \caption{\label{fig:tableau}A standard Young tableau of shape
      $(4,3,1)$.}
  \end{center}
\end{figure}

For a standard Young tableau $T$, recall the {\em descent set of $T$},
denoted $\Des(T)$, defined by
\begin{equation}
  \Des(T) = \{ (i,i+1) \; | \; i \; \mbox{lies strictly south of} \;
  i+1 \; \mbox{in} \; T \}.
\label{eqn:DesT}
\end{equation}
Completely analogous to the case with words, define the {\em major
  index of $T$}, denoted $\maj(T)$, by
\begin{equation}
  \maj(T) = \sum_{(i,i+k) \in \Des(T)} i .
\label{eqn:majT}
\end{equation}
For the example in \reffig{tableau}, $\Des = \{(1,2), (4,5), (7,8)\}$
and so $\maj = 1+4+7 = 12$. The descent set for tableaux corresponds
to the descent set of permutations in the sense that for a fixed set
$D$,
\begin{displaymath}
  \# \{ w \in \mathcal{S}_n \; | \; \Des(w) = D \} =
  \sum_{\lambda} f^{\lambda} \cdot  \# \{ T \in \SYT(\lambda) \; | \;
  \Des(T) = D \},
\end{displaymath}
where $\SYT(\lambda)$ denotes the set of standard Young tableaux of
shape $\lambda$ and $f^{\lambda} = |\SYT(\lambda)|$. This identity can
be proved using the Robinson-Schensted-Knuth correspondence which
bijectively associates each permutation $w$ with a pair of standard
tableaux $(P,Q)$ of the same shape such that $\iDes(w) = \Des(Q)$. We
postpone the definition of $\Des_k$ and $\maj_k$ for \refsec{tableau}.

\section{A family of bijections on words}
\label{sec:bijections}

For $k \geq 2$, we will construct bijections $\phi^{(k)}$ on words of
length $n$ such that 
\begin{equation}
  \maj_{k-1}(w) = \maj_k(\phi^{(k)}(w)) .
\label{eqn:dist}
\end{equation}
As noted earlier, these bijections are not equivalent to those defined
by Kadell, and the appropriate composition does not give Foata's
bijection. That said, the construction below follows the idea of
\cite{Foata1968} in that $\phi^{(k)}$ will be defined recursively
using an involution $\gamma_{j}^{(k)}$ which permutes the letters of a
given word. 

Let $x,a,b$ be (not necessarily distinct) integers. Say that $x$
{\em splits} the pair $a,b$ if $a \leq x < b$ or $b \leq x < a$. Let
$w$ be a word of length $n$. For $k \geq 2$ and $j \leq n$, define a
set of indices $\Gamma_{j}^{(k)}$ of $w$ by
\begin{equation}
  j-k \in \Gamma_{j}^{(k)}(w) \;\; \mbox{if} \;\; w_j
  \;\mbox{splits the pair} \; w_{j-k}, w_{j-k+1},
\label{eqn:Gamma-i}
\end{equation}
  and if $i \in \Gamma_{j}^{(k)}(w)$, then
\begin{equation}
  i-k \in \Gamma_{j}^{(k)}(w) \;\; \mbox{if exactly one of} \; w_{i}
  \;\mbox{or}\; w_{i+1} \;\mbox{splits the pair}\; w_{i-k},w_{i-k+1}.
\label{eqn:Gamma-r}
\end{equation}
For our running example, we have $\Gamma_{8}^{(3)}(9 \; 8 \; 6 \; 1 \;
7 \; 3 \; 2 \; 4 \; 5) = \{ 5, 2\}$.

Let permutations act on words by permuting the indices, i.e. $\tau
\cdot w \; = \; w_{\tau(1)} w_{\tau(2)} \cdots w_{\tau(n)}$. Define a
map $\gamma_{j}^{(k)}$ by
\begin{equation}
  \gamma_{j}^{(k)}(w) \; = \; \left( \prod_{i \in \Gamma_{j}^{(k)}(w)} (i,
  i+1) \right) \cdot w .
\label{eqn:gammak}
\end{equation}
That is to say, $\gamma_{j}^{(k)}(w)$ is the result of interchanging
$w_{i}$ and $w_{i+1}$ for all $i \in \Gamma_{j}^{(k)}(w)$. Back to our
running example, we have $\gamma_{8}^{(3)}(9 \; \mathbf{8} \;
\mathbf{6} \; 1 \; \mathbf{7} \; \mathbf{3} \; 2 \; 4 \; 5) = 9 \;
\mathbf{6} \; \mathbf{8} \; 1 \; \mathbf{3} \; \mathbf{7} \; 2 \; 4 \;
5$.

For $w$ a word of length $n$, define $\phi^{(k)}$ by
\begin{equation}
  \phi^{(k)} (w) \; = \; \gamma_{n}^{(k)} \circ \gamma_{n-1}^{(k)}
  \circ \cdots \circ \gamma_{1}^{(k)} (w).
\label{eqn:phik}
\end{equation}
Since $\gamma_j^{(k)}$ is the identity for $j \leq k$, these terms may
be omitted from equations \ref{eqn:phik} and \ref{eqn:psik}.

For example, for $w=6 \; 9 \; 3 \; 8 \; 1 \; 7 \; 2 \; 4 \; 5$,
$\phi^{(3)}(w)$ is computed as follows.
\begin{displaymath}
  \begin{array}{rcccccccccc}
    w \phantom{)} & = & 
      6 & 9 & 3 & 8 & \rn{m3}{1} & 7 & 2 & 4 & 5 \\
    \gamma_4^{(3)}(w) & = & 
      9 & 6 & 3 & 8 & \rn{m4}{1} & 7 & 2 & 4 & 5 \\
    \gamma_{5}^{(3)}\gamma_4^{(3)}(w) & = & 
      9 & 6 & 3 & 8 & \rn{m5}{1} & 7 & 2 & 4 & 5 \\
    \gamma_{6}^{(3)}\gamma_{5}^{(3)}\gamma_4^{(3)}(w) & = & 
      9 & 6 & 8 & 3 & \rn{m6}{1} & 7 & 2 & 4 & 5 \\
    \gamma_{7}^{(3)}\gamma_{6}^{(3)}\gamma_{5}^{(3)}\gamma_4^{(3)}(w) & = & 
      9 & 6 & 8 & 1 & \rn{m7}{3} & 7 & 2 & 4 & 5 \\
    \gamma_{8}^{(3)}\gamma_{7}^{(3)}\gamma_{6}^{(3)}\gamma_{5}^{(3)}\gamma_4^{(3)}(w) & = & 
      9 & 8 & 6 & 1 & \rn{m8}{7} & 3 & 2 & 4 & 5 \\
    \phi^{(3)}(w) = \gamma_{9}^{(3)}\gamma_{8}^{(3)}\gamma_{7}^{(3)}\gamma_{6}^{(3)}\gamma_{5}^{(3)}\gamma_4^{(3)}(w) & = & 
      9 & 8 & 6 & 1 & \rn{m9}{7} & 3 & 2 & 4 & 5
  \end{array}
  \psset{nodesep=3pt,linewidth=.1ex}
  \everypsbox{\scriptstyle}
\end{displaymath}
Notice that for this example $\maj_2(w) = 19 =
\maj_3(\phi^{(3)}(w))$. Before proving \refeq{dist} in general, we
take note of a few important properties that $\phi^{(k)}$ shares with
Foata's bijection (for Foata, properties (i) and (ii) are shown in
\cite{Foata1968}, and property (iii) is shown in \cite{FoSc1978}).

\begin{proposition}
  For each $k \geq 2$, we have
  \begin{itemize}
  \item[(i)] the map $\phi^{(k)}$ is a bijection on words on $M$ with
    fixed $\emptyset$ positions;
  \item[(ii)] for $w$ a word of length $n$, $w_{n-k+1} > w_n$ if and
    only if $\phi^{(k)}(w)_{n-k} > \phi^{(k)}(w)_n = w_n$;
  \item[(iii)] for $w$ a permutation, $\iDes(w) = \iDes(\phi^{(k)}(w))$.
  \end{itemize}
\label{prop:props}
\end{proposition}

\begin{proof}
  Since $\Gamma_{j}^{(k)}(\gamma_{j}^{(k)}(w)) = \Gamma_{j}^{(k)}(w)$,
  $\gamma_{j}^{(k)}$ is an involution on words of length $n$ for all
  $j \leq n$ and $k \geq 2$. Therefore $\phi^{(k)}$ is a bijection on
  words of length $n$ for all $k \geq 2$ with inverse given by
  \begin{equation}
    \psi^{(k)} (w) \; = \; \gamma_{1}^{(k)} \circ \cdots \circ
    \gamma_{n-1}^{(k)} \circ \gamma_{n}^{(k)} (w) .
  \label{eqn:psik}
  \end{equation}
  It is clear from the definition of $\gamma_j^{(k)}$ that
  $\phi^{(k)}$ in fact fixes the last $k-1$ letters of a word, so
  indeed the last letter is fixed for every $k$. Let $u =
  \gamma_{n-1}^{(k)} \cdots \gamma_{1}^{(k)}(w)$, $u_{j} = w_{j}$ for
  $j \geq n-k+1$. If $u_{n-k}$ and $u_{n-k+1}$ compare the same with
  $u_{n}$, then $u = \phi^{(k)}(w)$ and (ii) clearly holds; otherwise,
  these two letters are interchanged by $\gamma_{n}^{(k)}$, again
  showing that (ii) is satisfied. Also note that $\phi^{(k)}$ may be
  defined recursively by
  \begin{equation}
    \phi^{(k)}(wx) = \gamma_{n+1}^{(k)} \left( \phi^{(k)}(w) \right) x ,
  \label{eqn:recursivePhi}
  \end{equation}
  which completely parallels Foata's original construction.  Finally,
  since consecutive letters cannot be split, in the sense of
  $\Gamma_j^{(k)}$, they may never be interchanged by
  $\gamma_j^{(k)}$. Thus the inverse descent set is preserved.
\end{proof}

To prove \refeq{dist}, we follow the strategy of \cite{Foata1968}. The
key, therefore, lies in the following lemma.

\begin{lemma}
  For $k \geq 2$, $w$ a word of length $n$ and $j \leq n$,
  \begin{displaymath}
    \maj_k \left( \gamma_{j}^{(k)}(w_1 \cdots w_{j-1}) \right) 
    \; = \; \maj_k(w_1 \cdots w_{j-1}) + 
    \left\{ \begin{array}{rl}
      1 & \mbox{if} \;\; w_{j-k} > w_j \geq w_{j-k+1}, \\
     -1 & \mbox{if} \;\; w_{j-k+1} > w_j \geq w_{j-k}, \\
      0 & \mbox{otherwise}.
    \end{array} \right.
  \end{displaymath}
\label{lem:gamma}
\end{lemma}

\begin{proof}
  If neither of the first two cases holds, then $j-k \not\in
  \Gamma_{j}^{(k)}(w)$, so $\gamma_{j}^{(k)}(w) = w$ and the result is
  immediate. Assume, then, that $j-k \in \Gamma_{j}^{(k)}(w)$, and set
  $u = w_{j-k} w_{j-k+1} \cdots w_{j-1}$. Then
  \begin{equation}
    \maj_k\left(\gamma_{j}^{(k)}(u)\right) = 
    \maj_k(u) + \left\{ \begin{array}{rl}
        1 & \mbox{if} \;\; w_{j-k} > w_j \geq w_{j-k+1} , \\
        -1 & \mbox{if} \;\; w_{j-k+1} > w_j \geq w_{j-k} .
      \end{array} \right.
  \label{eqn:majk-u}
  \end{equation}
  For $i \in \Gamma_{j}^{(k)}(w)$, let $u = w_{i} w_{i+1} \cdots
  w_{j-1}$, and, by induction, assume that \refeq{majk-u} holds for
  $u$.  Let $u' = w_{i-k} w_{i-k+1} \cdots w_{j-1}$. We will show that
  $u'$ also satisfies \refeq{majk-u} by considering the contribution
  to $\maj_k$ of $w_{i-k},w_{i-k+1}, \ldots, w_{i-1}$. For $i-k+1 < h
  < i$, $k$-inversions and $k$-descents involving $w_h$ are the same
  for $u'$ and $\gamma_{j}^{(k)}(u')$, so we need only consider
  contributions from the potential $k$-inversions $(i-k,i-k+1)$ and
  $(i-k+1,i)$, and the potential $k$-descents $(i-k,i)$ and
  $(i-k+1,i+1)$.
  
  First suppose that $i-k \in \Gamma_{j}^{(k)}(w)$.  
  In all eight possible scenarios for $w_{i-k},w_{i-k+1},w_{i},w_{i+1}$,
  we have
  \begin{eqnarray*}
    (i-k,i-k+1) \in \Inv_k(w) & \Leftrightarrow & (i-k,i-k+1) \not\in
    \Inv_k\left(\gamma_{j}^{(k)}(w)\right), \\
    (i-k,i) \in \Des_k(w) & \Leftrightarrow & (i-k+1,i+1) \in
    \Des_k\left(\gamma_{j}^{(k)}(w)\right), \\
    (i-k+1,i+1) \in \Des_k(w) & \Leftrightarrow & (i-k,i) \in
    \Des_k\left(\gamma_{j}^{(k)}(w)\right).
  \end{eqnarray*}
  
  If both or neither of $(i-k,i)$ and $(i-k+1,i+1)$ are $k$-descents
  of $w$, then the same holds for $u'$ and $\gamma_{j}^{(k)}(u')$. In
  this case 
  exactly one of $(i-k,i-k+1)$ and $(i-k+1,i)$ is a $k$-inversion for
  $w$, and
  \begin{eqnarray*}
    (i-k+1,i) \in \Inv_k(u') & \Leftrightarrow & 
    (i-k+1,i) \not\in \Inv_k\left(\gamma_{j}^{(k)}(u')\right).
  \end{eqnarray*}
  The lemma now follows. On the other hand, if exactly one of
  $(i-k,i)$ and $(i-k+1,i+1)$ is a $k$-descent of $w$, 
  then the difference in the contribution to $\maj_k$ from the
  potential $k$-descents is offset by the difference from the
  potential $k$-inversion $(i-k,i-k+1)$. Furthermore,
  \begin{eqnarray*}
    (i-k+1,i) \in \Inv_k(u') & \Leftrightarrow & 
    (i-k+1,i) \in \Inv_k\left(\gamma_{j}^{(k)}(u')\right),
  \end{eqnarray*}
  thereby establishing the result.
  
  To complete the proof, note that when $i-k \not\in
  \Gamma_{j}^{(k)}(w)$, either $w_{i}$ and $w_{i+1}$ compare the same
  with $w_{i-k}$ and also with $w_{i-k+1}$ and so the $k$-inversions
  and $k$-descents beginning with $i-k$ or $i-k+1$ are unchanged, or
  the $k$-descent at $(i-k+1,i+1)$ is exchanged for a $k$-descent at
  $(i-k,i)$ along with a $k$-inversion at $(i,i+1)$. In both cases the
  contribution to the $k$-major index is preserved.
\end{proof}

\begin{proposition}
  For $k \geq 2$ and $w$ a word, $\displaystyle{\maj_{k-1} (w) =
    \maj_k \left( \phi^{(k)}(w) \right)}$.
\label{prop:majk}
\end{proposition}

\begin{proof}
  The result is clear for a words of length $\leq k$. We proceed by
  induction, assuming the result for words of length $n-1$.  Let $w$
  be a word of length $n-1$ and $x$ a letter. To simplify notation,
  let
  $$
  u = \gamma_{n}^{(k)} \left( \phi^{(k)}(w) \right).
  $$
  By expanding the definition of $\maj_k$ and applying
  \reflem{gamma}, we have
  \begin{displaymath}
    \begin{array}{l}
    \maj_k \left(\phi^{(k)}(wx)\right) \\
    \hspace{2em} = \maj_k(ux) \\ 
    \hspace{2em} = \maj_k(u) + \# \{ i > n-k \; | \; u_i > x \}
        + \left\{ \begin{array}{rl}
        0 & \mbox{if} \; x \geq u_{n-k} \\
        n\!-\!k & \mbox{if} \; u_{n-k} > x \end{array} \right. \\
    \hspace{2em} = \maj_k(u) + \# \{i>n-k+1 \; | \; u_i > x\} 
        + \left\{ \begin{array}{rcrl}
        n\!-\!k + 1 & \mbox{if} & u_{n-k} > x,    & u_{n-k+1} > x \\
        0       + 0 & \mbox{if} & x \geq u_{n-k}, & x \geq u_{n-k+1} \\
        n\!-\!k + 0 & \mbox{if} & u_{n-k} > x,    & x \geq u_{n-k+1} \\
        0       + 1 & \mbox{if} & x \geq u_{n-k}, & u_{n-k+1} > x
      \end{array} \right. \\
    \hspace{2em} = \maj_k \left( \gamma_{n}^{(k)}(u) \right)
    + \# \{i>n-k+1 \; | \; u_i > x\} 
        + \left\{ \begin{array}{rl}
        n\!-\!k\!+\!1 + 0 & \mbox{if} \; u_{n-k},u_{n-k+1} > x \\
        0             + 0 & \mbox{if} \; x \geq u_{n-k},u_{n-k+1} \\
        n\!-\!k\      + 1 & \mbox{if} \; u_{n-k} > x \geq u_{n-k+1} \\
        1             - 1 & \mbox{if} \; u_{n-k+1} > x \geq u_{n-k}
      \end{array} \right. \\
    \hspace{2em} = \maj_{k-1} \left( \phi^{(k)}(w) \right) 
          + \# \{i>n-k+1 \; | \; u_i > x\}
        + \left\{ \begin{array}{rl}
        0 & \mbox{if} \; x \geq u_{n-k} \\
        n\!-\!k\!+\!1 & \mbox{if} \; u_{n-k} > x
      \end{array} \right.
    \end{array}
  \end{displaymath}

  Recall from \refprop{props} that for $i \geq n-k+2$, $u_i = w_i$,
  and so
  $$
  \{i>n-k+1 \; | \; u_i > x\} \; = \; \{i>n-k+1 \; | \; w_i > x\} .
  $$
  Furthermore, since $\phi^{(k)}(w)_{n-k+1} = w_{n-k+1}$, we also have
  $$
  u_{n-k} \leq x \; \Leftrightarrow \; w_{n-k+1} \leq x .
  $$

  Continuing from the above equation using these two facts and the
  inductive hypothesis, we have
  \begin{displaymath}
    \maj_k \left(\phi^{(k)}(wx)\right) 
    = \maj_{k-1}(w) 
    + \# \{i>n-k+1 \; | \; w_i > x\} + \left\{ \begin{array}{rl}
        0 & \mbox{if} \; x \geq w_{n-k+1} \\
        n\!-\!k\!+\!1 & \mbox{if} \; w_{n-k+1} > x
      \end{array} \right.
  \end{displaymath}
  which is exactly $\maj_{k-1} (wx)$, as desired.
\end{proof}

For $1 \leq h < i$, we can compose these bijections to form the
bijection
\begin{equation}
  \phi^{[i,h]} = \phi^{(i)} \circ \cdots \circ \phi^{(h+1)}
\end{equation}
satisfying $\maj_h (w) = \maj_i \left( \phi^{[i,h]}(w) \right)$.  In
particular, $\phi^{[k,1]}$ provides a bijective proof of the
following.

\begin{theorem}
  Let $W_M$ be the set of words on a multiset $M$ with a fixed
  $\emptyset$ positions. Then for $k \geq 1$,
  \begin{eqnarray*}
    \sum_{w \in W_M} q^{\maj(w)} & = &
    \sum_{w \in W_M} q^{\maj_k(w)} .
  \end{eqnarray*}
  That is to say, the $k$-major index has Mahonian distribution.
\label{thm:mahonian}
\end{theorem}

\section{Extending the $k$-major index to tableaux}
\label{sec:tableau}

In \cite{HaSt2006}, Haglund and Stevens define an inversion number for
standard tableaux which is equidistributed with the major
index. Therefore it is natural to try to extend the $k$-major index
statistic to tableaux in a similar manner. However, to do this, we
must first define $\Des_k$ for standard Young tableaux.

Consider the possible relative positions of $i$ and $i+k$ in a
standard Young tableau $T$.  Since $i < i+k$, $i$ must lie strictly
west or strictly south of $i+k$. If $i$ lies strictly west and weakly
north of $i+k$, then the pair $(i,i+k)$ should not count as a
$k$-descent.  Conjugately, if $i$ lies strictly south and weakly east
of $i+k$, then the pair $(i,i+k)$ should count as a $k$-descent. The
difficulty arises in how to resolve the situation where $i$ lies
strictly southwest of $i+k$. The approach given in \cite{HaSt2006} is
quite involved as it is based on {\em inversion paths} which must be
computed iteratively. In most cases, interchanging even two
consecutive entries in a tableau completely alters the inversion paths
in an opaque way. Therefore we begin at the other extreme, though
below we succeed only up to $k=3$.

For $k=2$, the ambiguous case when $i$ lies strictly southwest of
$i+2$ cannot arise in a standard tableaux. However, for $k=3$ we must
decide whether $(i-3,i)$ is a $3$-descent when $i-3,i-2,i-1,i$ appear
in a $2\times 2$ box in $T$. For reasons that will be made clear, we
resolve the situations as indicated in \reffig{3Des}.

\begin{figure}[ht]
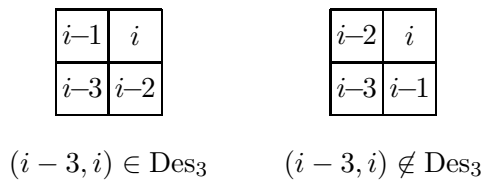

  \begin{center}
    \begin{displaymath}
      \begin{array}{\ccii\ccii}
        \bigtableau{i\!\!-\!\!1 & i \\ i\!\!-\!\!3 & i\!\!-\!\!2} & 
        \bigtableau{i\!\!-\!\!2 & i \\ i\!\!-\!\!3 & i\!\!-\!\!1} \\[2\vsp]
        (i-3,i) \in \Des_3 & (i-3,i) \not\in \Des_3
      \end{array}
    \end{displaymath}      
    \caption{\label{fig:3Des} Ambiguous cases for whether $(i-3,i)$
      should constitute a $3$-descent.}
  \end{center}    
\end{figure}

To simplify notation, we introduce the following terminology. For
$i<n$, say that {\em i attacks n} if $i$ lies strictly south and
weakly east of $n$ or if $i$ lies strictly southwest of $n$ and $i+1$
attacks $n$. 

\begin{definition}
  For $T$ a standard tableau, $k \leq 3$, define the {\em $k$-descent
    set of $T$}, denoted $\Des_{k}(T)$, by
  \begin{displaymath}
    \Des_k(T) = \{ (i,i+k) \; | \; \mbox{$i$ attacks $i+k$} \}, 
  \end{displaymath}  
  define the {\em set of $k$-inversions of $T$}, denoted
  $\Inv_k(T)$, by
  \begin{displaymath}
    \Inv_k(T) = \bigcup_{j < k} \Des_j(T) ,
  \end{displaymath}
  and finally define the {\em $k$-major index of $T$}, denoted
  $\maj_k(T)$, by
  \begin{displaymath}
    \maj_k (T) \; = \; \left| \Inv_k(T) \right| + \sum_{(i,i+k) \in
      \Des_k(T)} i .
  \end{displaymath}
\label{defn:Tstats}
\end{definition}

Note that for defining $k$-inversions we made use of the alternate
description of $k$-inversions for words given in \refeq{alt-Invk}. For
the example in \reffig{tableau}, we have $\Des_2 = \{(3,5), \ (4,6), \
(6,8)\}$, $\Inv_2 = \{(1,2), \ (4,5), \ (7,8)\}$ and so $\maj_2 = 3 +
3 + 4 + 6 = 16$.

Parallel to \refsec{bijections}, we aim to generalize
\refthm{mahonian} to tableaux by constructing bijections $\Phi^{(k)}$,
$k=2,3$, on standard Young tableaux of fixed shape such that
\begin{equation}
\maj_{k-1}(T) = \maj_k(\Phi^{(k)}(T)) .
\end{equation}

The first task, then, is to define the set $\Gamma_{j}^{(k)}$. Here
care must be taken when determining when the ``splitting'' condition
is satisfied. As a minimum requirement, since the intention is to
interchange $i$ and $i+1$, we must ensure that we do this only if $i$
and $i+1$ do not appear in the same row or column. This motivates the
decision in \reffig{3Des} as well as the following definitions. 

Say that $n$ {\em splits} $a,b$ if exactly one of $a,b$ attacks
$n$. For $k=2,3$, define $\Gamma_{j}^{(k)}$ by
\begin{displaymath}
  j-k \in \Gamma_{j}^{(k)}(T) \;\; \mbox{if} \;\; \mbox{$j$ splits the pair
    $j-k,j-k+1$},
\end{displaymath}
and if $i \in \Gamma_{j}^{(k)}(T)$, then
\begin{displaymath}
  i-k \in \Gamma_{j}^{(k)}(T) \;\; \mbox{if} \;\; \mbox{exactly one
    of}\; i,i+1 \; \mbox{splits the pair}\; i-k,i-k+1.
\end{displaymath}

By the definition of attacking, both or neither $i,i+1$ attack $n$
whenever $i$ is strictly southwest of $n$. Therefore in order for $n$
to split $i,i+1$, one must lie strictly south and weakly east of $n$,
and the other must lie weakly north of $n$. It follows, then, that if
$i$ and $i+1$ lie in the same row or column of $T$, then $n$ does not
split $i,i+1$ for any $n$. 

Let permutations act on standard fillings of a Young diagram by
permuting the entries. While this is not, in general, a well-defined
action on tableaux, the following application in fact is. For $k=2,3$,
define $\gamma_{j}^{(k)}$ by
\begin{equation}
  \gamma_{j}^{(k)}(T) \; = \; \left( \prod_{i \in \Gamma_{j}^{(k)}(T)}
    (i, i+1) \right) \cdot T .
\label{eqn:gammak-T}
\end{equation}
That is, $\gamma_{j}^{(k)}$ interchanges $i$ and $i+1$ for all $i \in
\Gamma_{j}^{(k)}(T)$. As before, $\gamma_j^{(k)}$ is the identity for
$j \leq k$.

If $i$ is strictly southwest of $n$, then $n$ cannot split the pair
$i,i+1$. It follows that $\Gamma^{(k)}_j(T) =
\Gamma^{(k)}_j(\gamma^{(k)}_j(T))$; in particular, $\gamma^{(k)}_j$ is
an involution. For $k=2,3$, define a bijection $\Phi^{(k)}$ on
tableaux of a fixed shape by
\begin{equation}
  \Phi^{(k)} (T) \; = \; \gamma_{n}^{(k)} \circ \gamma_{n-1}^{(k)}
  \circ \cdots \circ \gamma_{1}^{(k)} (T).
\label{eqn:phi2-T}
\end{equation}
For the example in \reffig{Phi2}, observe that $\maj_1(T) = 16 =
\maj_2(\Phi^{(2)}(T))$.

\begin{figure}[ht]
  \begin{center}
    \begin{displaymath}
      \begin{array}{\cciii\cciii c}
        T \ = \ \rn{l}{\tableau{8 \\ 2 & 4 & 6 \\ 1 & 3 & 5 & 7}} &
        \rn{c}{\tableau{8 \\ 3 & 4 & 6 \\ 1 & 2 & 5 & 7}} &
        \rn{r}{\tableau{8 \\ 2 & 5 & 6 \\ 1 & 3 & 4 & 7}} \ = \ \Phi^{(2)}(T)
      \end{array}
      \psset{nodesepA=3pt,nodesepB=5pt,linewidth=.1ex}
      \ncline{->}  {l}{c} \naput{\gamma^{(2)}_{4}}
      \ncline{->}  {c}{r} \naput{\gamma^{(2)}_{6}}
    \end{displaymath}
    \caption{\label{fig:Phi2}An example of $\Phi^{(2)}$; here
      $\gamma^{(2)}_{j} = \mathrm{id}$ for $j\neq 4,6$.}
  \end{center}    
\end{figure}

Similar to before, the inverse of $\Phi^{(k)}$ is given by composing
the maps $\gamma^{(k)}_j$ in the reverse order. This establishes the
analogue of the property (i) of \refprop{props}, and the analogue of
property (ii) is that the largest letter of $T$ is fixed by
$\Phi^{(k)}$. As property (iii) has no real analogue in this setting, we
move on to the more important statement observed in the example,
namely the analogue of \refprop{majk} below.

\begin{proposition}
  For $T$ a standard Young tableau and $k=2,3$, we have
  $\displaystyle{\maj_{k-1} (T) \; = \; \maj_k \left( \Phi^{(k)}(T)
    \right)}$.
\label{prop:majk-T}
\end{proposition}

\begin{proof}
  We use the proofs of \reflem{gamma} and \refprop{majk}. For this to
  make sense, we make the substitution that for $i < n$, $w_i > w_n$
  should be interpreted as ``i attacks n'' and similarly $w_i \leq
  w_n$ should be interpreted as ``i does not attack n''. In order for
  the arguments to remain valid under this translation, interchanging
  entries using $\gamma_j^{(k)}$ may not change $k$-inversions or
  $k$-descents between unmoved entries. The only potential violation
  of this is the potential $3$-descent between $i-3$ and $i$ in the
  situations depicted in \reffig{3Des}. However, in either case $i-2
  \not\in \Gamma^{(3)}_j$ since neither $i+1$ nor $i+2$ can split the
  pair $i-2, i-1$. Therefore, with this translation, the proofs carry
  through verbatim.
\end{proof}

\begin{theorem}
  For $\lambda$ a partition, we have
  \begin{equation}
    \sum_{T \in \SYT(\lambda)} q^{\maj(T)} =
    \sum_{T \in \SYT(\lambda)} q^{\maj_{2}(T)} =
    \sum_{T \in \SYT(\lambda)} q^{\maj_{3}(T)} .
  \end{equation}
  \label{thm:mahonian-T}
\end{theorem}

Unfortunately, \refthm{mahonian-T} is the best we can do towards
extending \refthm{mahonian} using this direct analogue of
$\phi^{(k)}$. This technique breaks down at $k=4$ for the shape
$(2,2,2)$. In this case, the $6$ must lie in the northeast corner and
will necessarily interchange the $2$ and $3$ if they both lie in the
first two rows. Then if $1,2,3,4$ occupy the first two rows, this
changes whether $(1,4)$ is a $3$-descent ($4$-inversion). In order to
overcome this shortfall, either we must adopt a more dynamic notion of
$k$-inversions as in the Haglund-Stevens approach or a more
complicated bijection.

\section{Connections with Macdonald polynomials}
\label{sec:macdonald}

The $k$-major index statistic was rediscovered in the author's study
of Macdonald polynomials. In this section we connect the results of
\refsec{bijections} back to Macdonald polynomials.

In \cite{HHLRU2005}, Bylund and Haiman introduced the $k$-inversion
number of a $k$-tuple of tableaux to be the number of inversions
between certain entries, and it is shown that this statistic may be
used to give an alternative definition for Lascoux-Leclerc-Thibon
polynomials. In particular, when each shape of the $k$-tuple is a
ribbon, i.e. contains no $2 \times 2$ block, the $k$-inversion number
of the $k$-tuple is exactly $|\Inv_k(w)|$ where $w$ is a certain
reading word of the $k$-tuple. In further study of these objects
\cite{Assaf2007-2}, it became natural to associate to each $k$-tuple
not only the $k$-inversion number, but also a $k$-descent set. Again,
when the shapes of the $k$-tuple in question are all ribbons, this is
exactly given by $\Des_k(w)$ for the same reading word $w$. Here it is
essential that $w$ be allowed to contain $\emptyset$'s in order to
correctly space the entries of the $k$-tuple.

The case when the $k$-tuple consists entirely of ribbons is an
important special case in light of \cite{Haglund2004,HHL2005} where it
is shown that the Macdonald polynomials are in fact positive sums of
LLT polynomials where the shapes are ribbons. In this context, the
index $k$ is given by the number of columns of the indexing partition
of the Macdonald polynomial. The Macdonald Positivity Theorem,
conjectured by Macdonald in 1988 \cite{Macdonald1988}, was first
proved by Haiman using algebraic geometry \cite{Haiman2001}, and more
recently by Grojnowski and Haiman using Kazhdan-Lusztig theory
\cite{GrHa2007} and the author using a purely combinatorial argument
\cite{Assaf2007-2}. This latter proof, while purely combinatorial,
relies on new combinatorial machinery, namely {\em dual equivalence
  graphs}, involving rather technical proofs of the main
theorems. Below we suggest how Macdonald positivity may be recovered
in a completely elementary way using bijections similar to
$\phi^{(k)}$.

The main idea behind \cite{Assaf2007-2} is to group together terms of
a Macdonald polynomial which contribute to a single Schur function and
have the same associated statistics. This is done in three steps; for
complete details, see \cite{Assaf2007-3}. First, quasisymmetric
functions are used to reduce to standard words, i.e. permutations, and
it is here that the inverse descent set of a permutation is
relevant. Next, for a given $k$, the permutations are divided into
equivalence classes (in the language of \cite{Assaf2007-2}, connected
components of a graph) using the following involutions.

For $i \geq 2$, define involutions $d_i$ and $\tilde{d}_i$ on permutations
where $i$ does not lie between $i-1$ and $i+1$ by
\begin{eqnarray}
  d_i    (\cdots\;   i  \;\cdots\;i\pm 1\;\cdots\;i\mp 1\;\cdots ) 
  & = &  \cdots\;i\mp 1\;\cdots\;i\pm 1\;\cdots\;  i   \;\cdots \; ,
  \label{eqn:d} \\
  \tilde{d}_i(\cdots\;   i  \;\cdots\;i\pm 1\;\cdots\;i\mp 1\;\cdots ) 
  & = &  \cdots\;i\pm 1\;\cdots\;i\mp 1\;\cdots\;  i   \;\cdots \; ,
  \label{eqn:dwig}
\end{eqnarray}
where all other entries remain fixed. Combining these, define
$D^{(k)}_i$ by
\begin{equation}
  D^{(k)}_i(w) \; = \; \left\{
    \begin{array}{ll}
      d_i(w) & \mbox{if} \;\; \dist(i-1,i,i+1) > k \\
      \tilde{d}_i(w) & \mbox{if} \;\; \dist(i-1,i,i+1) \leq k
    \end{array} \right. ,
  \label{eqn:Dk}
\end{equation}
where $\dist(i-1,i,i+1)$ is the maximum distance between the positions
of $i-1,i,i+1$ in $w$. The $\emptyset$'s, or spacers, in $w$ are
essential for this step as they adjust the relative distance of the
letters of $w$.

\begin{definition}
  Call two permutations $w$ and $u$ {\em $k$-equivalent}, denoted
  $w \sim_{k} u$, if $w = D^{(k)}_{i_1} D^{(k)}_{i_2} \cdots
  D^{(k)}_{i_m}(u)$ for some sequence $i_1,i_2,\ldots,i_m \geq 2$.
\end{definition}

\begin{figure}[ht]
  \begin{center}
    \begin{displaymath}
      \begin{array}{\ccr \ccr \ccr \ccr c}
        \mbox{$1$-classes:} &
        \left\{ 1 \ 2 \ 3 \right\}; &
        \left\{ 2 \ 1 \ 3 \ , \ 3 \ 1 \ 2 \right\}; &
        \left\{ 2 \ 3 \ 1 \ , \ 1 \ 3 \ 2 \right\}; &
        \left\{ 3 \ 2 \ 1 \right\} \\[.3\vsp]
        \mbox{$2$-classes:} &
        \left\{ 1 \ 2 \ 3 \right\}; &
        \left\{ 2 \ 1 \ 3 \ , \ 1 \ 3 \ 2 \right\}; &
        \left\{ 2 \ 3 \ 1 \ , \ 3 \ 1 \ 2 \right\}; &
        \left\{ 3 \ 2 \ 1 \right\}
      \end{array}
    \end{displaymath}      
    \caption{\label{fig:classes} Equivalence classes of permutations
      of length $3$.}
  \end{center}
\end{figure}

\begin{remark}
  Note that the $1$-equivalence classes are exactly the {\em dual
    equivalence classes} for partitions; see \cite{Haiman1992}. In
  particular, the sum of the quasisymmetric functions associated to
  the permutations in a $1$-equivalence class is a Schur function.
\label{rmk:dec}
\end{remark}

A key observation in \cite{Assaf2007-2} is that $\Des_k(w) =
\Des_k\left( D^{(k)}_{i} (w) \right)$ and $|\Inv_k(w)| = |\Inv_k\left(
  D^{(k)}_{i} (w) \right)|$.  In particular, $\Des_k$ and $|\Inv_k|$
are constant on $k$-equivalence classes. Therefore, the third and
final step toward establishing the Macdonald Positivity Theorem is to
prove that the sum over the quasisymmetric functions associated to a
given $k$-equivalence class is Schur positive. By \refrmk{dec}, a
natural approach is to relate $k$-classes to $1$-classes. Indeed, the
proof presented in \cite{Assaf2007-2} does this by showing that a
connected component of the graph for $k$-columns (a $k$-equivalence
class) may be broken into a union of connected dual equivalence graphs
($1$-equivalence classes). It is for this step that the proof becomes
quite technical and involved, and so the idea is to bypass the
machinery of dual equivalence graphs altogether. The following
proposition achieves this for the $2$-column/$2$-equivalence class
case.

\begin{proposition}
  For $w$ a permutation such that $i$ does not lie between $i-1$ and
  $i+1$, we have
  \begin{equation}
    \phi^{(2)} \left( D^{(1)}_i(w) \right) = D^{(2)}_i \left(
      \phi^{(2)}(w) \right) .
  \end{equation}
\label{prop:2equiv}
\end{proposition}

\begin{proof}
  First note that $D^{(1)}_i = d_i$. Furthermore, $D^{(2)}_i(w) =
  d_i(w)$ unless $i-1,i,i+1$ are adjacent in $w$. Without loss of
  generality, we may assume that $w_r = i+1$, $w_s = i-1$ and $w_t =
  i$ for some indices $r<s<t$. Set $\ww = d_1(w)$, so that $\ww_r =
  i$, $\ww_s = i-1$ and $\ww_t = i+1$. We aim to show that
  $D^{(2)}_i(\phi^{(2)}(w)) = \phi^{(2)}(\ww)$.

  For notational convenience, we write $\gamma_j$ for $\gamma^{(2)}_j$
  and $\Gamma_j$ for $\Gamma^{(2)}_j$. Since the definition of
  $\Gamma_j$ depends only on the relative orders of letters, it
  follows that $\Gamma_{j} (w) = \Gamma_{j}(\ww)$ for $j < t$. Along
  the same lines, $\Gamma_{t}(w) \neq \Gamma_{t}(\ww)$ if and only if
  $r,s = t-2,t-1$. If this is not the case, then $\gamma_t 
  \cdots  \gamma_1 (w) = d_i \left( \gamma_t  \cdots 
    \gamma_1 (\ww) \right)$ and indeed $\dist(i-1,i,i+1) > 2$ in
  $\gamma_t  \cdots  \gamma_1 (w)$. In the affirmative case,
  $\Gamma_{t}(w) = \{t-2\}$ since both or neither $i-1,i+1$ splits any
  pair of preceding letters and $\Gamma_{t}(\ww) =
  \emptyset$. Therefore $\gamma_t  \cdots  \gamma_1 (w) =
  \tilde{d}_i \left( \gamma_t  \cdots  \gamma_1 (\ww)
  \right)$ as desired since $\dist(i-1,i,i+1) = 2$ in $\gamma_t 
  \cdots  \gamma_1 (w)$.

  Now consider the effect of $\gamma_j$ for $j > t$. For the same
  reasons as before, $\Gamma_{j}(w) \neq \Gamma_{j}(\ww)$ if and only
  if $i-1,i,i+1$ are adjacent either before or after $\gamma_j$ is
  applied. For $\ww$, the relative positions of $i-1,i,i+1$ will never
  change. Moreover, the position of $i+1$ in $\ww$ tracks the position
  of $i$ in $w$, and the positions of $i,i-1$ in $\ww$ are the
  positions of $i-1,i+1$ in $w$ (though not necessarily
  respectively). For $w$, each time $i$ moves between adjacent and
  nonadjacent to $i-1,i+1$, the difference between $\Gamma_j$ for $w$
  and $\ww$ is exactly that the former contains the index of the
  leftmost of $i-1,i+1$ and the latter does not. Comparing $d_i$ with
  $\tilde{d}_i$, this is exactly the difference between the two
  involutions, i.e. $i-1$ and $i+1$ interchange positions. Therefore
  in the end, $\tilde{d}_i(\phi^{(2)}(w)) = \phi^{(2)}(\ww)$ if
  $i-1,i,i+1$ are adjacent in $\phi^{(2)}(w)$, and $d_i(\phi^{(2)}(w))
  = \phi^{(2)}(\ww)$ otherwise.
\end{proof}

Recall that the sum over of an equivalence class is determined by the
quasisymmetric functions associated to the permutation of the
class. Since the quasisymmetric function associated to a permutation
is determined by the inverse descent set of the permutation,
\refprop{props} (iii) and \refrmk{dec} establish the following
corollary to \refprop{2equiv}.

\begin{corollary}
  Macdonald polynomials indexed by partitions with $2$ columns are
  Schur positive.
\label{cor:2pos}
\end{corollary}

For $k \geq 3$, it is not possible for $\phi^{(k)} \left(
  D^{(k-1)}_i(w) \right) = D^{(k)}_i \left( \phi^{(k)}(w) \right)$ in
general. The reason for this is that the sizes of the $k$-equivalence
classes increase with $k$. For permutations of length $n$, the
$n$-equivalence classes have a nice description given in
\cite{Assaf2007-2} which allows us to prove the following.

\begin{proposition}
  For $w$ a permutation such that $i$ does not lie between $i-1$ and
  $i+1$, we have
  \begin{equation}
    \phi^{[1,n]} (w) \sim_{n} \phi^{[1,n]} \left( D^{(1)}_i(w) \right) .
  \end{equation}
\label{prop:nequiv}
\end{proposition}

\begin{proof}
  For this case, $\maj_n = \inv$ in the usual sense and there are no
  $n$-descents to consider. As already noted, $\inv$ is constant on
  $n$-equivalence classes and, since $D^{(n)}_i \equiv \tilde{d}_i$,
  $w_1 > w_n$ for some $w$ in an $n$-class if and only if $w_1 > w_n$
  for every $w$ in an $n$-class. Furthermore, it is not difficult to
  show that these two properties completely characterize $n$-classes.
  Since $D^{(1)}_i \equiv d_i$, $w_{n-1} > w_n$ for some $w$ in a
  $1$-class if and only if $w_{n-1} > w_n$ for every $w$ in a
  $1$-class.  By \refprop{majk}, $\maj(w) = \inv(\phi^{[1,n]} (w))$,
  and by \refprop{props} (ii), $w_{n-1} > w_n$ if and only if
  $\phi^{[1,n]} (w)_1 > \phi^{[1,n]} (w)_n$. Therefore if $w \sim_1
  u$, then $\inv(\phi^{[1,n]} (w)) = \inv(\phi^{[1,n]} (u))$ and
  $\phi^{[1,n]} (w)_1 > \phi^{[1,n]} (w)_n$ if and only if
  $\phi^{[1,n]} (u)_1 > \phi^{[1,n]} (u)_n$. The result now follows.
\end{proof}

\begin{corollary}
  Macdonald polynomials indexed by a single row are Schur positive.
\label{cor:npos}
\end{corollary}

Given this, one might still hope to express each $k$-equivalence class
as a union of the images of certain $k-1$-equivalence classes under an
appropriate map. However, for $k \geq 3$, neither $\phi^{(k)}$ nor the
corresponding composition of Kadell's bijections accomplishes
this. There is, however, considerable evidence suggesting that such a
family of bijections does exist, and so we conclude with the following
conjecture which, as a corollary, would yield a simple proof of
Macdonald positivity.

\begin{conjecture}
  There exists a family of bijections $\theta^{(k)}$ on permutations
  satisfying Propositions \ref{prop:props} and \ref{prop:majk} such
  that if $w \sim_{k-1} u$ then $\theta^{(k)}(w) \sim_{k}
  \theta^{(k)}(u)$.
\label{conj:theta}
\end{conjecture}

%
%

\bibliographystyle{amsalpha}

%
%

\bibliographystyle{abbrv} 
\bibliography{majk}

\end{document}